\documentclass[12pt]{article}

\usepackage{amsmath}
\usepackage{amssymb}
\usepackage[mathscr]{eucal}
\usepackage{graphicx}
\usepackage{color}

\usepackage{bm}
\usepackage{mathrsfs}
\usepackage{amsthm}
\usepackage{enumerate}
\usepackage[mathscr]{eucal}
\usepackage{graphicx}

\usepackage{tikz}
\usepackage{pstricks}
\usepackage{pst-node}
\usepackage{pst-plot}
\usetikzlibrary{shapes,snakes}

\newtheorem{Theorem}{\sc Theorem}

\newtheorem{Lemma}[Theorem]{\sc Lemma}

\newtheorem{Remark}[Theorem]{\sc Remark}

\newcommand{\R}{{\if mm {\rm I}\mkern -3mu{\rm R}\else \leavevmode
		\hbox{I}\kern -.17em\hbox{R} \fi}}

\def\sqr#1#2{{
		\vcenter{
			\vbox{\hrule height.#2pt
				\hbox{\vrule width.#2pt height#1pt \kern#1pt
					\vrule width.#2pt
				}
				\hrule height.#2pt
			}
		}
}}

\def\bar{\overline}
\def\real{\mathbb{R}}

\def\lista#1
{{ \itemindent 0.0cm \labelsep .2cm \leftmargin 0.8cm \rightmargin
		0.0cm \labelwidth 0.6cm \topsep 0.0mm
		\parsep 0.0mm
		\itemsep 0.0mm
		\begin{list}{}
			{ \setlength{\leftmargin}{.8cm} \setlength{\rightmargin}{0.0cm}
				\setlength{\parsep}{0.0mm} \setlength{\topsep}{.0mm}
				\setlength{\parskip}{.0cm} \setlength{\itemsep}{.0cm} }
			{#1}\end{list}} }

\begin{document}
\title{Numerical analysis of a family of simultaneous distributed-boundary mixed elliptic optimal control problems and their asymptotic behaviour through a commutative diagram and error estimates}

\author{
Carolina M. Bollo  \footnote{\, Depto. Matem\'atica, FCEFQyN, Universidad Nacional de R\'io Cuarto, Ruta 36 Km 601,
5800 R\'io Cuarto, Argentina. E-mail: cbollo@exa.unrc.edu.ar., cgariboldi@exa.unrc.edu.ar.} \ Claudia M. Gariboldi $^{*}$ 
\ Domingo A. Tarzia \footnote{\, Depto. Matem\'atica, FCE, Universidad Austral, Paraguay 1950, S2000FZF Rosario, Argentina.} \footnote{\, CONICET, Argentina. E-mail: DTarzia@austral.edu.ar.}
}

\date{}

\maketitle
	
\medskip

\noindent {\bf Abstract.} \

In this paper, we consider a family of simultaneous distributed-boundary optimal control problems ($P_{\alpha}$) on the internal energy and the heat flux for a system governed by a mixed elliptic variational equality with a parameter $\alpha >0$ (the
heat transfer coefficient on a portion of the boundary of the domain) and a simultaneous distributed-boundary optimal control problem ($P$) governed also by an elliptic variational equality with a Dirichlet boundary condition on the same portion of the boundary. We formulate discrete approximations $\left(P_{h \alpha}\right)$ and $\left(P_h\right)$ of the optimal control problems $\left(P_\alpha\right)$ and $(P)$ respectively, for each $h>0$ and for each $\alpha>0$, through the finite element method with Lagrange's triangles of type 1 with parameter $h$ (the longest side of the triangles). The goal of this paper is to study the convergence of this family of discrete simultaneous distributed-boundary mixed elliptic optimal control problems $\left(P_{h \alpha}\right)$ when the parameters $\alpha$ goes to infinity and the parameter $h$ goes to zero simultaneously. We prove the convergence of the family of discrete problems $\left(P_{h \alpha}\right)$ to the discrete problem $\left(P_h\right)$ when $\alpha \rightarrow +\infty$, for each $h>0$, in adequate functional spaces. We study the convergence of the discrete problems $\left(P_{h \alpha}\right)$ and $\left(P_h\right)$, for each $\alpha >0$, when $h \rightarrow 0^+$ obtaining a commutative diagram which relates the continuous and discrete simultaneous distributed-boundary mixed elliptic optimal control problems $\left(P_{h \alpha}\right),\left(P_\alpha\right),\left(P_h\right)$ and $(P)$ by taking the limits $h \rightarrow 0^+$ and $\alpha \rightarrow +\infty$ respectively. We also study the double convergence of $\left(P_{h \alpha}\right)$ to $(P)$ when $(h, \alpha) \rightarrow(0^+,+\infty)$ which represents the diagonal convergence in the above commutative diagram.

\medskip
	
\medskip
	

\medskip

{\thispagestyle{empty}} 

\section{Introduction}

We consider a bounded domain $\Omega$ in $\real^d$ whose
regular boundary $\Gamma $ consists of the union of two disjoint portions $\Gamma_{i}$, $i=1$, $2$,
with $|\Gamma_{i}|>0$, where $|\Gamma_i|$ denotes the $(d-1)$-dimensional
Hausdorff measure of the portion $\Gamma_i$ on $\Gamma$.
The outward normal vector on
the boundary is denoted by $n$.
We formulate the following classical steady-state heat conduction problems
with mixed boundary conditions ~\cite{AK, BBP, G, LCB, Ta3}:
\begin{eqnarray}
&&
-\Delta u=g \ \ \mbox{in} \ \ \Omega,
\ \ \quad u\big|_{\Gamma_{1}}=b,
\ \ \quad-\frac{\partial u}{\partial n}\big|_{\Gamma_{2}}=q,
\label{P}
\end{eqnarray}
\begin{equation}\label{Palfa}
-\Delta u=g \ \ \mbox{in} \ \ \Omega,
\ \quad -\frac{\partial u}{\partial n}\big|_{\Gamma_{1}}=\alpha(u-b),
\ \quad  -\frac{\partial u}{\partial n}\big|_{\Gamma_{2}}=q,  
\end{equation}
where $u$ is the temperature in $\Omega$,
$g$ is the internal energy in $\Omega$,
$b=Const.>0$ is the temperature on $\Gamma_{1}$ for the system (\ref{P}) and the temperature of
the external neighborhood on $\Gamma_{1}$ for the system (\ref{Palfa}) respectively,
$q$ is the heat flux on $\Gamma_{2}$ and $\alpha>0$ is the heat transfer
coefficient on $\Gamma_{1}$,
which satisfy the hypothesis:
$g\in H=L^2(\Omega)$ and 
$q\in Q=L^2(\Gamma_2)$.

Throughout the paper we use the following notation:
\begin{eqnarray*}
&& V=H^{1}(\Omega),
\quad
V_{0}=\{v\in V / v = 0 \ \ \mbox{on} \ \ \Gamma_{1} \},
\\[2mm]
&&
K=\{v\in V /
v = b \ \ \mbox{on} \ \ \Gamma_{1} \}=b+V_0,
\\
&&
a(u,v)=\int_{\Omega }\nabla u \, \nabla v \, dx, \quad L(v)= \int_{\Omega}g v \,dx -
\int_{\Gamma_{2}}q \gamma (v) \,d\Gamma,\\
&&
a_{\alpha}(u,v)=a(u,v)+\alpha\int_{\Gamma_{1}} \gamma (u)  \gamma (v) \, d\Gamma, \quad L_{\alpha}(v)= L(v)+\alpha\int_{\Gamma_{1}}b \gamma (v) \, d\Gamma,
\end{eqnarray*}
where $\gamma \colon V \to L^2(\Gamma)$
denotes the trace operator on $\Gamma$.
In what follows, we write $u$ for the trace of a function $u \in V$ on the boundary.
In a standard way, we obtain the following variational formulations of (\ref{P}) and (\ref{Palfa}), \cite{Ta}:
\begin{eqnarray}
&&
\hspace{-1cm}
\mbox{find} \ \ u\in K \ \ \mbox{such that}\ \
a(u,v)=L(v)
\ \ \mbox{for all} \ \ v\in V_{0},
\label{Pvariacional}
\end{eqnarray}
\begin{eqnarray}
&&
\hspace{-1cm}
\mbox{find} \ \ u_{\alpha}\in V \ \ \mbox{such that}\ \
a_{\alpha}(u_{\alpha},v)=L_{\alpha}(v)
\ \ \mbox{for all} \ \ v\in V.
\label{Palfavariacional}
\end{eqnarray}
The standard norms on $V$ and $V_0$ are denoted by
\begin{eqnarray*}
&&
\| v \|_V = \Big(
\| v \|^2_{L^2(\Omega)}
+ \| \nabla v \|^2_{L^2(\Omega;\real^d)} \Big)^{1/2}
\ \ \mbox{for} \ \ v \in V, \\ [2mm]
&&
\| v \|_{V_0} = \| \nabla v \|_{L^2(\Omega;\real^d)}
\ \ \mbox{for} \ \ v \in V_0.
\end{eqnarray*}
It is well known by the Poincar\'e inequality,
see~\cite{CLM, R}, that on $V_0$ the above two norms
are equivalent. Note that the bilinear, symmetric and continuous forms $a$ and $a_{\alpha}$ are coercive on $V_{0}$ and $V$ respectively, that is, \cite{KS}:
\begin{equation}\label{coercive}
\exists \lambda >0\quad \text{such that}\quad
a(v, v) = \|v\|^{2}_{V_0} \ge \lambda \|v\|^{2}_{V}
\ \ \mbox{for all} \ \ v\in V_{0},
\end{equation}
\begin{equation}\label{coercivealfa}
\exists \lambda_{\alpha}>0\quad \text{such that}\quad
a_{\alpha}(v, v) = \|v\|^{2}_{V_0} \ge \lambda_{\alpha} \|v\|^{2}_{V}
\ \ \mbox{for all} \ \ v\in V
\end{equation}
where $\lambda_{\alpha}=\lambda_{1}\min\{1,\alpha\}$, with $\lambda_{1}>0$ the coerciveness constant for the bilinear form $a_{1}$, \cite{KS,TT}.

We remark that, under additional hypotheses on the data
$g$, $q$ and $b$, problem (\ref{P}) can be considered as steady-state two-phase Stefan problem, see ~\cite{GT,TT,Ta, Ta3}.

We consider the following continuous optimal control problems \cite{GaTa2,Li,Tr}:
\begin{itemize}
\item [($P$)] A simultaneous distributed and Neumann boundary optimal control problem, given by:
\begin{equation}\label{OP}
\text{find}\quad (\overline{g},\overline{q}) \in H \times Q \quad \text{such that} \quad J((\overline{g},\overline{q}))=\min_{(g,q)\in H\times Q }J(g,q)
\end{equation}
with
\begin{equation}\label{JOP}
J(g,q)=\frac{1}{2}||u_{gq}-z_{d}||^{2}_{H}+\frac{M_1}{2}||g||^{2}_{H}+\frac{M_2}{2}||q||^{2}_{Q}
\end{equation}
where $u_{gq}$ is the unique solution to the variational equality (\ref{Pvariacional}) for $g\in H$ and $q\in Q$, $z_{d}\in H$ given and $M_1$ and $M_2$ are positive constants given. 
\item [$(P_{\alpha}$)] For each $\alpha>0$, the simultaneous distributed and Neumann boundary optimal control problem:
\begin{equation}\label{OPalfa}
\text{find}\quad (\overline{g}_{\alpha},\overline{q}_{\alpha})\in H\times Q \quad \text{such that} \quad J_{\alpha}(\overline{g}_{\alpha},\overline{q}_{\alpha})=\min_{(g,q)\in H\times Q}J_{\alpha}(g,q)
\end{equation}
with
\begin{equation}\label{JOPalfa}
J_{\alpha}(g,q)=\frac{1}{2}||u_{\alpha gq}-z_{d}||^{2}_{H}+\frac{M_1}{2}||g||^{2}_{H}+\frac{M_2}{2}||q||^{2}_{Q}
\end{equation}
where $u_{\alpha gq}$ is a solution to the variational equality (\ref{Palfavariacional}) for $g\in H$, $q\in Q$ and $\alpha>0$, $z_{d}\in H$ is given and $M_1$ and $M_2$ are positive constants.
\end{itemize}

In relation with the simultaneous optimal control problems (\ref{OP}) and (\ref{OPalfa}), we define the adjoint states, as the unique solutions of the variational equalities, \cite{GaTa2}:
\begin{eqnarray}
&&
\hspace{-1cm}
\mbox{find} \ \ p_{gq}\in V_{0} \ \ \mbox{such that}\ \
a(p_{gq},v)=(u_{gq}-z_{d},v)_{H}
\ \ \mbox{for all} \ \ v\in V_{0},
\label{Padjvariacional}
\end{eqnarray}
\begin{eqnarray}
&&
\hspace{-1cm}
\mbox{find} \ \ p_{\alpha gq}\in V \ \ \mbox{such that}\ \
a_{\alpha}(p_{\alpha gq},v)=(u_{\alpha gq}-z_{d},v)_{H}
\ \ \mbox{for all} \ \ v\in V.
\label{Palfaadjvariacional}
\end{eqnarray}

The unique continuous simultaneous vectorial optimal controls $(\overline{g},\overline{q})$ and $(\overline{g}_{\alpha},\overline{q}_{\alpha})$ can be characterized, following \cite{GaTa, GaTa2}, as a fixed point on $H\times Q$ for suitable operators $W$ and $W_\alpha$ over their optimal adjoint system states $p_{\overline{g}\,\overline{q}} \in V_0$ and $p_{\alpha\overline{g}_{\alpha}\,\overline{q}_{\alpha}} \in V$, defined by:
$$
\begin{aligned}
&W: H\times Q \rightarrow H\times Q \quad \text {such that} \quad W(g,q)=(-\frac{1}{M_1} p_{gq},\frac{1}{M_2} p_{gq})\\
&W_\alpha: H\times Q \rightarrow H\times Q \quad \text {such that} \quad W_\alpha(g,q)=(-\frac{1}{M_1} p_{\alpha gq},\frac{1}{M_2} p_{\alpha gq}).
\end{aligned}
$$
The limit of the optimal control problems (\ref{OPalfa}) when $\alpha \rightarrow +\infty$ was studied in \cite{GaTa2} and it was proved that:
$$
\lim _{\alpha \rightarrow +\infty}\left\|u_{\alpha \overline{g}_{\alpha}\,\overline{q}_{\alpha}}-u_{\overline{g}\,\overline{q}}\right\|_V=0, \quad \lim _{\alpha \rightarrow +\infty}\left\|p_{\alpha \overline{g}_{\alpha}\,\overline{q}_{\alpha}}-p_{\overline{g}\,\overline{q}}\right\|_V=0, 
$$
$$
\lim _{\alpha \rightarrow +\infty}\left\|(\overline{g}_{\alpha},\overline{q}_{\alpha} )-(\overline{g},\overline{q})\right\|_{H\times Q}=0
$$
where the norm in $H\times Q$ is defined by:
$$
||(g,q)||_{H\times Q}^{2}=||g)||_{H}^{2}+||q||_{Q}^{2},\quad \forall (g.q)\in H\times Q.
$$

Now, we consider the finite element method and a polygonal domain $\Omega \subset \mathbb{R}^n$ with a regular triangulation with Lagrange triangles of type 1, constituted by affine-equivalent finite element of class $C^0$ being $h$ the parameter of the finite element approximation which goes to zero \cite{BS,Ci}. Then, we discretize the elliptic variational equalities for the system states (\ref{Pvariacional}) and (\ref{Palfavariacional}), the adjoint system states (\ref{Padjvariacional}) and (\ref{Palfaadjvariacional}), and the cost functional (\ref{JOP}) and (\ref{JOPalfa}), respectively. In general, the solution of a mixed elliptic boundary problem belongs to $H^r(\Omega)$ with $1<r \leq 3 / 2-\epsilon$ $(\epsilon>0)$, but there exist some examples which solutions belong to $H^r(\Omega)$ with $2 \leq r$ \cite{AK,LCB,S}. 

The goal of this paper is to study the numerical analysis, by using the finite element method, of the convergence results corresponding to the continuous simultaneous distibuted-boundary elliptic optimal control problems (\ref{OP}) and (\ref{OPalfa}) when $\alpha \rightarrow +\infty$. Moreover, the following commutative diagram which relates the continuous simultaneous distributed-boundary mixed optimal control problems $(P_{\alpha})$ and $(P)$, with the discrete simultaneous distributed-boundary mixed optimal control problems $(P_{h\alpha})$ and $(P_h)$ is obtained by taking the limits $h\rightarrow 0^+$, $\alpha \rightarrow +\infty$ and $(h,\alpha)\rightarrow (0^+,+\infty)$ as follows:

\begin{center}
\begin{tikzpicture}
   \draw (0,5) node {\textbf{Problem} \boldsymbol{$(P_{\alpha})$}};
   \draw (0,4.3) node {$(\overline{g}_{\alpha},\overline{q}_{\alpha}),\,u_{\alpha\overline{g}_{\alpha}\overline{q}_{\alpha}},\,p_{\alpha\overline{g}_{\alpha}\overline{q}_{\alpha}}$};
   \draw (9,5) node {\textbf{Problem} \boldsymbol{$(P)$}};
   \draw (9,4.3) node {$(\overline{g},\overline{q}),\,u_{\overline{g}\,\overline{q}},\,p_{\overline{g}\,\overline{q}}$};
   \draw (0,0) node {\textbf{Problem} \boldsymbol{$(P_{h\alpha})$}};
     \draw (0,0.7) node {$(\overline{g}_{h\alpha},\overline{q}_{h\alpha}),\,u_{h\alpha\overline{g}_{h\alpha}\overline{q}_{h\alpha}},\,p_{h\alpha\overline{g}_{h\alpha}\overline{q}_{h\alpha}}$};
   \draw (9,0) node {\textbf{Problem} \boldsymbol{$(P_{h})$}};
    \draw (9,0.7) node {$(\overline{g}_{h},\overline{q}_{h}),\,u_{h\overline{g}_{h}\overline{q}_{h}},\,p_{h\overline{g}_{h}\overline{q}_{h}}$};
   \draw[-latex] (0,1.5) to (0,3.5);
    \draw (1,2.5) node {$h\rightarrow 0^+$};
   \draw[-latex] (9,1.5) to (9,3.5);
    \draw (10,2.5) node {$h\rightarrow 0^+$};
      \draw[-latex] (3.5,0.5) to (6.5,0.5);
      \draw (5,1) node {$\alpha\rightarrow +\infty$};
       \draw[-latex] (3.5,4.5) to (6.5,4.5);
  \draw (5,5) node {$\alpha\rightarrow +\infty$};
    \draw[-latex] (1.7,1.3) to (7.5,3.9);
     \draw (6.6,2.5) node {$(h,\alpha)\rightarrow (0^+, +\infty)$};
\end{tikzpicture}
\end{center}
where $(\overline{g}_h, \overline{q}_h)$, $u_{h\overline{g}_h\overline{q}_h}$ and $p_{h \overline{g}_h\overline{q}_h}$ are the optimal control, system state and adjoint state of the discrete simultaneous distributed-boundary optimal control problem $(P_{h})$ for each $h>0$, and $(\overline{g}_{h\alpha}, \overline{q}_{h\alpha})$, $u_{h\alpha \overline{g}_{h\alpha}\overline{q}_{h\alpha}}$ and $p_{h\alpha \overline{g}_{h\alpha}\overline{q}_{h\alpha}}$ are the optimal control, the system state and adjoint state of the discrete simultaneous distributed-boundary optimal control problem $(P_{h\alpha})$ for each $h>0$ and $\alpha >0$, respectively.

The study of the limit $h\rightarrow 0^{+}$ of the discrete solutions of optimal control problems can be considered as a classical limit, see \cite{CM1, CM2, CR, DGH, DH, HMRS, HH, H1, HM, MH, YCY, YCL} but the double limit $(h,\alpha)\rightarrow (0^{+},+\infty)$ can be considered as a new ones for a vectorial control problem.

The paper is structured as follows. In Section 2, we formulate the discrete elliptic variational equalities for the system states $u_{hgq}$ and $u_{h\alpha gq}$, we define the discrete cost functional $J_h$ and $J_{h\alpha}$, we formulate the discrete simultaneous distributed-boundary optimal control problems $(P_h)$ and $(P_{h\alpha})$, and the discrete elliptic variational equalities for the adjoint states $p_{hgq}$ and $p_{h\alpha gq}$ for each $\alpha >0$ and $h>0$. We obtain properties for the discrete optimal control problems and we define contraction operators $W_h$ and $W_{h\alpha}$ which allows obtain the optimal controls $(\overline{g}_h, \overline{q}_h)$ and $(\overline{g}_{h\alpha}, \overline{q}_{h\alpha})$ as fixed points. In Section 3, we study the convergences of the discrete
optimal control problems $(P_h)$ to $(P)$, and $(P_{h\alpha})$ to $(P_{\alpha})$ when $h\rightarrow 0^{+}$ (for each $\alpha >0$). In Section 4, we study the convergence of the discrete optimal control problems $(P_{h\alpha})$ to $(P_h)$ when $\alpha \rightarrow +\infty$ (for each $h>0$) and we obtain a commutative diagram which relates the continuous and discrete optimal control problems by taking the limits $h\rightarrow 0^{+}$ and $\alpha \rightarrow +\infty$. In Section 5, we study the double convergence of the discrete optimal control problems
$(P_{h\alpha})$ to $(P)$ when $(h,\alpha)\rightarrow (0^{+},+\infty)$ and we obtain the diagonal convergence in the previous commutative diagram. In Section 6, we obtain the relationship and estimations among the optimal values $J(\overline{g},\overline{q})$, $J(\overline{g}_h,\overline{q}_h)$, $J_h(\overline{g}_h,\overline{q}_h)$ and $J_h(\overline{g},\overline{q})$ corresponding to the optimal control problems ($P$) and ($P_h$) and the same estimations corresponding to the optimal control problems ($P_{\alpha}$) and ($P_{h\alpha}$). In Section 7, we formulate the conclusions of this paper.

\section{Discretization by finite element method and pro\-perties}\label{Discretization}

In this section, we consider the finite element method and a polygonal domain $\Omega \subset \mathbb{R}^n$ with a regular triangulation with Lagrange triangles of type 1, constituted by affine-equivalent finite element of class $C^0$ being $h$ the parameter of the finite element approximation which goes to zero \cite{BS,Ci}. We can take $h$ equal to the longest side of the triangles $T \in \tau_h$ and we can approximate the sets $V, V_0$ and $K$ by:
$$
V_h=\left\{v_h \in C^0(\bar{\Omega}) / v_h |_{T} \in P_1(T), \forall T \in \tau_h\right\},
$$
$$
V_{0 h}=\left\{v_h \in V_h / v_h =0 \,\,\text{on}\,\, \Gamma_1\right\}, \quad K_h=b+V_{0 h}
$$
where $P_1$ is the set of the polymonials of degree less than or equal to 1. Let $\pi_h:C^0(\bar{\Omega}) \rightarrow V_h$ be the corresponding linear interpolation operator. Then there exists a constant $c_0>0$ (independent of $h$) such that $\forall v \in H^r(\Omega)$, $1<r \leq 2$, \cite{BS}:
\begin{equation}\label{Estim1}
\left\|v-\pi_h(v)\right\|_H \leq c_0 h^r\|v\|_r
\end{equation} 
\begin{equation}\label{Estim2}
\left\|v-\pi_h(v)\right\|_V \leq c_0 h^{r-1}\|v\|_r.
\end{equation} 
The discrete cost functional $J_h, J_{h \alpha}: H\times Q \rightarrow \mathbb{R}_0^{+}$ are defined by:
\begin{equation}\label{Jh}
J_h(g,q)=\frac{1}{2}\left\|u_{h gq}-z_d\right\|_H^2+\frac{M_1}{2}\|g\|_H^2+\frac{M_2}{2}\|q\|_Q^2
\end{equation}
\begin{equation}\label{Jhalfa}
J_{h \alpha}(g,q)=\frac{1}{2}\left\|u_{h \alpha gq}-z_d\right\|_H^2+\frac{M_1}{2}\|g\|_H^2+\frac{M_2}{2}\|q\|_Q^2.
\end{equation}
where $u_{h gq}$ and $u_{h\alpha gq}$ are the discrete system states defined as the solution of the following discrete elliptic variational equalities \cite{KS, Ta1}:
\begin{equation}\label{Variacuh}
u_{h gq} \in K_h: \quad a\left(u_{h gq}, v_h\right)=\left(g, v_h\right)_H- (q, v_h)_Q, \quad \forall v_h \in V_{0 h}, 
\end{equation}
\begin{equation}\label{Variacuhalfa}
u_{h \alpha gq} \in V_h: \quad a_\alpha\left(u_{h \alpha gq}, v_h\right)=\left(g, v_h\right)_H- (q,v_h)_Q +\alpha \int_{\Gamma_1} b v_h d \gamma, \quad \forall v_h \in V_h.
\end{equation}
The corresponding discrete distributed optimal control problems consists in finding $(\overline{g}_{h},\overline{q}_{h}), (\overline{g}_{h \alpha},\overline{q}_{h \alpha}) \in H\times Q$ such that:
\begin{equation}\label{Ph}
\text {Problem}\left(P_h\right): \quad J_h\left(\overline{g}_{h},\overline{q}_{h}\right)=\underset{(g,q) \in H\times Q}{\operatorname{Min}} J_h(g,q),
\end{equation}
\begin{equation}\label{Phalfa}
\text{Problem}\left(P_{h \alpha}\right): \quad J_{h \alpha}\left(\overline{g}_{h \alpha},\overline{q}_{h \alpha}\right)=\underset{(g,q) \in H\times Q}{\operatorname{Min}} J_{h\alpha}(g,q)
\end{equation}
and their corresponding discrete adjoint states $p_{h gq}$ and $p_{h \alpha gq}$ are defined respectively as the solution of the following discrete elliptic variational equalities:
\begin{equation}\label{Variacph}
p_{h gq} \in V_{0 h}: \quad  a\left(p_{h gq}, v_h\right)=\left(u_{h gq}-z_d, v_h\right)_H, \quad \forall v_h \in V_{0 h} 
\end{equation}
\begin{equation}\label{Variacphalfa}
p_{h \alpha gq} \in V_h: \quad a_\alpha\left(p_{h \alpha gq}, v_h\right)=\left(u_{h \alpha gq}-z_d, v_h\right)_H, \quad \forall v_h \in V_h.
\end{equation}
\begin{Remark} We note that the discrete (in the d-dimensional space) distributed optimal control problem $\left(P_h\right)$ and $\left(P_{h \alpha}\right)$ are still an infinite dimensional optimal control problem since the control space is not discretized.
\end{Remark}
\begin{Lemma}
\begin{itemize}
\item [(i)] For all $(g,q) \in H\times Q$, $b>0$ on $\Gamma_1$, there exist unique solutions $u_{h gq} \in K_h$ and $p_{h gq} \in V_{0 h}$ of the elliptic variational equalities (\ref{Variacuh}) and (\ref{Variacph}) respectively, and $u_{h\alpha gq} \in V_h$ and $p_{h\alpha gq} \in V_h$ of the elliptic variational equalities (\ref{Variacuhalfa}) and (\ref{Variacphalfa}), respectively.
\item [(ii)] The operators $(g,q) \in H\times Q \rightarrow u_{h gq} \in V$, and $(g,q) \in H\times Q \rightarrow u_{h\alpha g q} \in V$ are Lipschitzians, i.e., $\forall (g_1,q_1), (g_2,q_2) \in H\times Q, \forall h>0$
$$
\left\|u_{h g_2 q_2}-u_{h g_1 q_1}\right\|_V \leq \frac{(1+||\gamma||)\sqrt{2}}{\lambda}\left\|(g_2,q_2)-(g_1,q_1)\right\|_{H\times Q},
$$
$$
\left\|u_{h \alpha g_2 q_2}-u_{h \alpha g_1 q_1}\right\|_V \leq \frac{(1+||\gamma||)\sqrt{2}}{\lambda_\alpha}\left\|(g_2,q_2)-(g_1,q_1)\right\|_{H\times Q}. 
$$
where $||\gamma||$ is the norm of the trace operator.
\item [(iii)] We have, $\forall (f,\eta)\in H\times Q$ the following equalities:
$$
a(p_{hgq},u_{hf\eta}-u_{h00})=(f,p_{hgq})_H -(\eta, p_{hgq})_Q
$$
$$
a_{\alpha}(p_{h\alpha gq},u_{h\alpha f\eta}-u_{h\alpha 00})=(f,p_{h\alpha gq})_H -(\eta, p_{h\alpha gq})_Q
$$
where $u_{h00}$ and $u_{h\alpha 00}$ are the unique solutions for data $g=0$ and $q=0$, to the problems (\ref{Variacuh}) and (\ref{Variacuhalfa}), respectively.
\item [(iv)] The operators $(g,q) \in H\times Q \rightarrow p_{h gq} \in V_{0 h}$, and $(g,q) \in H \times Q\rightarrow p_{h\alpha gq} \in V_h$ are Lipschitzians and strictly monotones, i.e., $\forall (g_1,q_1), (g_2,q_2) \in H\times Q, \forall h>0$, we have:
$$
\text {a) } \left(p_{hg_2 q_2}-p_{hg_1 q_1}, g_2-g_1\right)_H - \left(p_{hg_2 q_2}-p_{hg_1 q_1}, q_2-q_1\right)_Q=\left\|u_{h g_2 q_2}-u_{h g_1 q_1}\right\|_H^2 \geq 0,
$$
\begin{equation*}
\begin{split}
 \text {b)} \left(p_{h\alpha g_2 q_2}-p_{h\alpha g_1 q_1}, g_2-g_1\right)_H - \left(p_{h\alpha g_2 q_2}-p_{h\alpha g_1 q_1}, q_2-q_1\right)_Q &= \left\|u_{h \alpha g_2 q_2}-u_{h\alpha g_1 q_1}\right\|_H^2 \\ & \geq 0, 
\end{split}
\end{equation*}
$$
\text {c) } \left\|p_{h g_2 q_2}-p_{h g_1 q_1}\right\|_V \leq \frac{(1+||\gamma||)\sqrt{2}}{\lambda^2}\left\|(g_2,q_2)-(g_1,q_1)\right\|_{H\times Q},
$$
$$
 \text {d) }\left\|p_{h\alpha g_2 q_2}-p_{h\alpha g_1 q_1}\right\|_V \leq \frac{(1+||\gamma||)\sqrt{2}}{\lambda_{\alpha}^2}\left\|(g_2,q_2)-(g_1,q_1)\right\|_{H\times Q}. 
$$
\end{itemize}
\end{Lemma}
\begin{proof}
We use the Lax-Milgram Theorem, the variational equalities (\ref{Variacuh}), (\ref{Variacuhalfa}), (\ref{Variacph}) and (\ref{Variacphalfa}), the coerciveness (\ref{coercive}) and (\ref{coercivealfa}) and following \cite{GaTa2, Li, Ta4, Ta5}.
\end{proof}
\begin{Theorem} 
\begin{itemize}
\item [(i)] The discrete cost functional $J_h$ and $J_{h \alpha}$ are $H$-elliptic and strictly convex applications, that is, $\forall (g_1,q_1), (g_2,q_2) \in H\times Q, \forall t \in[0,1]$, we have:
$$
\begin{gathered}
(1-t) J_h\left(g_2,q_2\right)+t J_h\left(g_1,q_1\right)-J_h\left((1-t)(g_2,q_2)+t(g_1,q_1)\right)\\
=\frac{t(1-t)}{2}\left\|u_{h g_2 q_2}-u_{h g_1 q_1}\right\|_H^2+M_1 \frac{t(1-t)}{2}\left\|g_2-g_1\right\|_H^2 + M_2 \frac{t(1-t)}{2}\left\|q_2-q_1\right\|_Q^2\\
\geq m \frac{t(1-t)}{2}\left\|(g_2,q_2)-(g_1,q_1)\right\|_{H\times Q}^2, 
\end{gathered}
$$
and
$$
\begin{gathered}
(1-t) J_{h \alpha}\left(g_2,q_2\right)+t J_{h \alpha}\left(g_1,q_1\right)-J_{h \alpha}\left((1-t)(g_2,q_2)+t(g_1,q_1)\right)\\
=\frac{t(1-t)}{2}\left\|u_{h \alpha g_2 q_2}-u_{h \alpha g_1 q_1}\right\|_H^2+M_1 \frac{t(1-t)}{2}\left\|g_2-g_1\right\|_H^2 +M_2 \frac{t(1-t)}{2}\left\|q_2-q_1\right\|_Q^2 \\
\geq m \frac{t(1-t)}{2}\left\|(g_2,q_2)-(g_1,q_1)\right\|_{H\times Q}^2, 
\end{gathered}
$$
where $m=\min\{M_1,M_2\}$.
\item [(ii)] There exist a unique optimal control $(\overline{g}_h,\overline{q}_h) \in H\times Q$ and $(\overline{g}_{h\alpha},\overline{q}_{h\alpha}) \in H\times Q$ that satisfy the optimization problems (\ref{Ph}) and (\ref{Phalfa}), respectively.
\item [(iii)] $J_h$ and $J_{h\alpha}$ are Gâteaux differenciable applications and their derivatives are given by the following expressions, $\forall (f,\eta) \in H\times Q$, $\forall h>0$:
$$
J_h^{\prime}(g,q)(f-g,\eta -q)=(f-g,p_{hgq}+M_1 g)_H+(\eta -q,M_2 q -p_{hgq})_Q,
$$
$$
J_{h\alpha}^{\prime}(g,q)(f-g,\eta -q)=(f-g,p_{h\alpha gq}+M_1 g)_H+(\eta -q,M_2 q -p_{h\alpha gq})_Q.
$$
\item [(iv)] The optimality condition for the problems (\ref{Ph}) and (\ref{Phalfa}) are given by, $\forall (f,\eta)\in H\times Q$:
$$
J_h^{\prime}(\overline{g}_h,\overline{q}_h)(f,\eta)=0 \Leftrightarrow (f,p_{h\overline{g}_h\,\overline{q}_h}+M_1 \overline{g}_h)_H+(\eta,M_2 \overline{q}_h -p_{h\overline{g}_h\,\overline{q}_h})_Q=0
$$
$$J_{h\alpha}^{\prime}(\overline{g}_{h\alpha},\overline{q}_{h\alpha})(f,\eta)=0 \Leftrightarrow (f,p_{h\overline{g}_{h\alpha}\,\overline{q}_{h\alpha}}+M_1 \overline{g}_{h\alpha})_H+(\eta ,M_2 \overline{q}_{h\alpha} -p_{h\overline{g}_{h\alpha}\,\overline{q}_{h\alpha}})_Q=0.
$$
\item [(v)] $J_h^{\prime}$ and $J_{h \alpha}^{\prime}$ are Lipschitzian and strictly monotone operators, i.e., $\forall (g_1,q_1), (g_2,q_2) \in H\times Q, \forall h>0$, we have:
$$
\left\|J_h^{\prime}(g_2,q_2)-J_h^{\prime}(g_1,q_1)\right\|_{H\times Q} \leq\left(M+\frac{(1+||\gamma||)^2}{\lambda^2}\right)\sqrt{2}\left\|(g_2,q_2)-(g_1,q_1)\right\|_{H\times Q},
$$
\begin{equation*}
\begin{split}
\left\langle J_h^{\prime}(g_2,q_2)-J_h^{\prime}(g_1,q_1), (g_2,q_2)-(g_1,q_1)\right\rangle &=\left\|u_{h g_2 q_2}-u_{h g_1 q_1}\right\|_H^2\\ & +M_1\left\|g_2-g_1\right\|_H^2 +M_2\left\|q_2-q_1\right\|_Q^2 \\ & \geq m\left\|(g_2,q_2)-(g_1,q_1)\right\|_{H\times Q}^2,
\end{split}
\end{equation*}
$$
\left\|J_{h \alpha}^{\prime}(g_2,q_2)-J_{h\alpha}^{\prime}(g_1,q_1)\right\|_{H\times Q} \leq\left(M+\frac{(1+||\gamma||)^2}{\lambda_{\alpha}^2}\right)\sqrt{2}\left\|(g_2,q_2)-(g_1,q_1)\right\|_{H\times Q},
$$
\begin{equation*}
\begin{split}
\left\langle J_{h\alpha}^{\prime}(g_2,q_2)-J_{h\alpha}^{\prime}(g_1,q_1),(g_2,q_2)-(g_1,q_1)\right\rangle &=\left\|u_{h\alpha g_2 q_2}-u_{h\alpha g_1 q_1}\right\|_H^2 \\& +M_1\left\|g_2-g_1\right\|_H^2 +M_2\left\|q_2-q_1\right\|_Q^2 \\ & \geq m\left\|(g_2,q_2)-(g_1,q_1)\right\|_{H\times Q}^2
\end{split}
\end{equation*}
where $M=\max\{M_1,M_2\}$ and $m=\min\{M_1,M_2\}$.
\end{itemize}
\end{Theorem}
\begin{proof}
We use the definitions (\ref{Jh}) and (\ref{Jhalfa}), the elliptic variational equalities (\ref{Variacuh}) and (\ref{Variacuhalfa}) and the coerciveness (\ref{coercive}) and (\ref{coercivealfa}), following \cite{GaTa1, GaTa2, Li, Ta4, Ta5}. 
\end{proof}
We define the operators
\begin{equation*}
W_h: H\times Q \rightarrow V_{0 h} \times Q \subset V_0 \times  Q\subset H \times Q \quad \text{such that} 
\end{equation*}
\begin{equation}\label{Wh}
W_h(g,q)=(-\frac{1}{M_1} p_{hgq},\frac{1}{M_2} \gamma (p_{hgq})) 
\end{equation}
\begin{equation*}
W_{h\alpha}: H\times Q \rightarrow V_h \times Q \subset V \times Q \subset H \times Q \quad \text{such that}
\end{equation*}
\begin{equation}\label{Whalfa}
W_{h \alpha}(g,q)=(-\frac{1}{M_1} p_{h\alpha gq},\frac{1}{M_2} \gamma (p_{h\alpha gq})).
\end{equation}
and we prove the following result.
\begin{Theorem} We have that:
\begin{itemize}
\item [(i)] $W_h$ and $W_{h \alpha}$ are Lipschitzian operators, that is, $\forall (g_1,q_1), (g_2,q_2) \in H\times Q, h>0$:
$$\left\|W_h(g_2,q_2)-W_h(g_1,q_1)\right\|_{H\times Q} \leq C_0\left\|(g_2,q_2)-(g_1,q_1)\right\|_{H\times Q},$$ 
$$\left\|W_{h \alpha}(g_2,q_2)-W_{h\alpha}(g_1,q_1)\right\|_{H\times Q} \leq C_{0\alpha}\left\|(g_2,q_2)-(g_1,q_1)\right\|_{H\times Q}$$
with $C_0=\frac{\sqrt{2}}{\lambda^2}\sqrt{\frac{1}{M_1^2}+\frac{||\gamma||^2}{M_2^2}}(1+||\gamma||)$ and $C_{0\alpha}=\frac{\sqrt{2}}{\lambda_{\alpha}^2}\sqrt{\frac{1}{M_1^2}+\frac{||\gamma||^2}{M_2^2}}(1+||\gamma||)$.
\item [(ii)] $W_h\left(W_{h\alpha}\right)$ is a contraction operator if and only if $C_0 <1$ ($C_{0\alpha}<1$).
\item [(iii)] If data satisfy inequality $C_0 <1$ ($C_{0\alpha}<1$), then the unique solution $(\overline{g}_h,\overline{q}_h)$ ($(\overline{g}_{h\alpha},\overline{q}_{h\alpha})$) to the discrete optimal control $P_h$ ($P_{h\alpha}$) can be obtained as the unique fixed point of the operator $W_h\left(W_{h\alpha}\right)$, that is:
$$
W_h(\overline{g}_h,\overline{q}_h)=(\overline{g}_h,\overline{q}_h)\quad \text{and}\quad W_{h\alpha}(\overline{g}_{h\alpha},\overline{q}_{h\alpha})=(\overline{g}_{h\alpha},\overline{q}_{h\alpha}).
$$
\end{itemize}
\end{Theorem}
\begin{proof} 
This results by using the definitions (\ref{Wh}) and (\ref{Whalfa}), and following \cite{GaTa2}.
\end{proof}

\section{Convergence of the discrete distributed-boundary optimal control problems ($P_h$) to ($P$), and ($P_{h\alpha}$) to ($P_{\alpha}$)
 when $h\rightarrow 0^{+}$}\label{Convergenceh}

In this section, we obtain error estimates between the optimal controls, system and adjoint states of the discrete simultaneous distributed-boundary optimal control problems $(P_h)$ and $(P_{h\alpha})$ and convergence results of the discrete optimal control problems $(P_h)$ to $(P)$ and $(P_{h\alpha})$ to $(P_{\alpha})$ when $h\rightarrow 0^{+}$, for each $\alpha >0$.
\begin{Lemma}\label{lem1}
(i) If the continuous system states and the continuous adjoint states have the regularity $u_{gq}$, $u_{\alpha g q}$, $p_{gq}$, $p_{\alpha gq} \in H^r(\Omega)$ $(1<r \leq 2)$, then $\forall \alpha>0$, $\forall (g,q) \in H\times Q$, $h>0$, we have the following estimations:
\begin{equation}\label{Est1}
\left\|u_{gq}-u_{h gq}\right\|_V \leq \frac{c_0}{\sqrt{\lambda}}\left\|u_{gq}\right\|_r h^{r-1}, \quad\left\|p_{gq}-p_{h gq}\right\|_V \leq c_1 h^{r-1}
\end{equation}
\begin{equation}\label{Est2}
\left\|u_{h \alpha gq}-u_{\alpha gq}\right\|_V \leq c_{0\alpha} h^{r-1}, \quad\left\|p_{h \alpha gq}-p_{\alpha gq}\right\|_V \leq c_{1\alpha} h^{r-1}
\end{equation}
where $c_0$ (given in (\ref{Estim1}) and (\ref{Estim2})), $c_1$, $c_{0\alpha}$ and $c_{1\alpha}$ are constants independents of $h$.
\newline
(ii) We have the following convergences, $\forall (g,q) \in H\times Q$:
$$
\lim _{h \rightarrow 0^{+}}\left\|u_{gq}-u_{h gq}\right\|_V=0, \quad \lim _{h \rightarrow 0^{+}}\left\|p_{gq}-p_{h gq}\right\|_V=0, 
$$ 
$$
\lim _{h \rightarrow 0^{+}}\left\|u_{h\alpha gq}-u_{\alpha gq}\right\|_V=0, \quad \lim _{h \rightarrow 0^{+}}\left\|p_{h \alpha gq}-p_{\alpha gq}\right\|_V=0, \quad \forall \alpha>0.
$$
\end{Lemma}
\begin{proof}
By using the variational equalities (\ref{Pvariacional}), (\ref{Palfavariacional}), (\ref{Padjvariacional}), (\ref{Palfaadjvariacional}), (\ref{Variacuh}), (\ref{Variacuhalfa}), (\ref{Variacph}) and (\ref{Variacphalfa}), the coerciveness properties (\ref{coercive}) and (\ref{coercivealfa}), the estimations (\ref{Estim1}) and (\ref{Estim2}) and the following properties, $\forall (g,q) \in H\times Q$:
$$
a\left(p_{gq}-p_{h gq}, \pi_h\left(p_{gq}\right)-p_{h gq}\right)=\left(u_{gq}-u_{h gq}, \pi_h\left(p_{gq}\right)-p_{h gq}\right)
$$
$$
a_{\alpha}\left(p_{\alpha gq}-p_{h \alpha gq}, \pi_h\left(p_{\alpha gq}\right)-p_{h \alpha gq}\right)=\left(u_{h \alpha gq}-u_{\alpha gq}, \pi_h\left(p_{\alpha gq}\right)-p_{h \alpha gq}\right)
$$
following a similar method given in \cite{Ta4, Ta5}, the thesis holds.
\end{proof}
\begin{Theorem}
We consider the continuous system states and adjoint states have the regularities $u_{\overline{g}\,\overline{q}}, u_{\alpha \overline{g}_{\alpha}\overline{q}_{\alpha}}, p_{\overline{g}\,\overline{q}}, p_{\alpha \overline{g}_{\alpha}\overline{q}_{\alpha}} \in H^r(\Omega)$ $(1<r \leq 2)$:
\begin{itemize}
\item [i)] We have the following limits, $\forall \alpha>1$:
\begin{equation}\label{Lim1}
\lim _{h \rightarrow 0^{+}}\left\|(\overline{g}_{h},\overline{q}_{h})-(\overline{g},\overline{q})\right\|_{H\times Q}=0
\end{equation}
\begin{equation}\label{Lim2}
\lim _{h \rightarrow 0^{+}}\left\|u_{h \overline{g}_{h}\overline{q}_{h}}-u_{\overline{g}\,\overline{q}}\right\|_V=0, \quad \lim _{h \rightarrow 0^{+}}\left\|p_{h \overline{g}_{h}\overline{q}_{h}}-p_{\overline{g}\,\overline{q}}\right\|_V=0 
\end{equation}
\begin{equation}\label{Lim3}
\lim _{h \rightarrow 0^{+}}\left\|(\overline{g}_{h\alpha},\overline{q}_{h\alpha})-(\overline{g}_{\alpha},\overline{q}_{\alpha})\right\|_{H\times Q}=0
\end{equation}
\begin{equation}\label{Lim4}
\lim _{h \rightarrow 0^{+}}\left\|u_{h \alpha\overline{g}_{h\alpha}\overline{q}_{h\alpha}}-u_{\alpha\overline{g}_{\alpha}\,\overline{q}_{\alpha}}\right\|_V=0, \quad \lim _{h \rightarrow 0^{+}}\left\|p_{h \alpha\overline{g}_{h\alpha}\overline{q}_{h\alpha}}-p_{\alpha\overline{g}_{\alpha}\overline{q}_{\alpha}}\right\|_V=0.
\end{equation}
\item [ii)] If data $M_1$ and $M_2$ satisfy the following inequalities 
\begin{equation}\label{puntofijo}
\frac{\sqrt{2}}{\lambda^2}\sqrt{\frac{1}{M_1^2}+\frac{||\gamma||^2}{M_2^2}}(1+||\gamma||)<1 \quad \text{and}\quad \frac{\sqrt{2}}{\lambda_{\alpha}^2}\sqrt{\frac{1}{M_1^2}+\frac{||\gamma||^2}{M_2^2}}(1+||\gamma||)<1
\end{equation} 
we have the following error bonds:
\begin{equation}\label{cotaop1}
\left\|(\overline{g}_{h},\overline{q}_{h})-(\overline{g},\overline{q})\right\|_{H\times Q} \leq c h^{r-1}
\end{equation}
\begin{equation}\label{cotaop2}
\left\|u_{h \overline{g}_{h}\overline{q}_{h}}-u_{\overline{g}\,\overline{q}}\right\|_V \leq c h^{r-1},\quad \left\|p_{h \overline{g}_{h}\overline{q}_{h}}-p_{\overline{g}\,\overline{q}}\right\|_V \leq c h^{r-1}
\end{equation}
\begin{equation}\label{cotaop3}
\left\|(\overline{g}_{h\alpha},\overline{q}_{h\alpha})-(\overline{g}_{\alpha},\overline{q}_{\alpha})\right\|_{H\times Q} \leq c_{\alpha} h^{r-1}
\end{equation}
\begin{equation}\label{cotaop4}
\left\|u_{h \alpha\overline{g}_{h\alpha}\overline{q}_{h\alpha}}-u_{\alpha\overline{g}_{\alpha}\,\overline{q}_{\alpha}}\right\|_V \leq c_{\alpha} h^{r-1}, \quad\left\|p_{h \alpha\overline{g}_{h\alpha}\overline{q}_{h\alpha}}-p_{\alpha\overline{g}_{\alpha}\overline{q}_{\alpha}}\right\|_V \leq c_{\alpha} h^{r-1}
\end{equation}
where $c$ and $c_{\alpha}$ are different constants independents of $h$.
\end{itemize}
\end{Theorem}
\begin{proof}
We follow a similar method to the one developed in \cite{Ta4, Ta5}. 

(i) From the definition of the functional (\ref{Jh}), we obtain, $\forall h>0$:
$$
\frac{1}{2}\left\|u_{h \overline{g}_{h}\overline{q}_{h}}-z_d\right\|_H^2+\frac{M_1}{2}\|\overline{g}_{h}\|_H^2+\frac{M_2}{2}\|\overline{q}_{h}\|_Q^2\leq \frac{1}{2}\left\|u_{h 00}-z_d\right\|_H^2\leq c
$$
where $u_{h00}$ is the unique solution of the variational equality (\ref{Variacuh}) for $g=0$ and $q=0$. That is,
$$
\left\|u_{h \overline{g}_{h}\overline{q}_{h}}\right\|_H\leq c\quad \|\overline{g}_{h}\|_H\leq c\quad \text{and}\quad \|\overline{q}_{h}\|_Q\leq c
$$
with $c$ different positive constants independent of $h$. Moreover, by using the variational equality (\ref{Variacuh}), we obtain
$$
\left\|u_{h \overline{g}_{h}\overline{q}_{h}}-b\right\|_V\leq \frac{1}{\lambda}(\|\overline{g}_{h}\|_H+\|\overline{q}_{h}\|_Q \|\gamma\|)\leq c
$$
then
$$
\left\|u_{h \overline{g}_{h}\overline{q}_{h}}\right\|_V\leq c.
$$
Next, by using the variational equality (\ref{Variacph}), we have
$$
\left\|p_{h \overline{g}_{h}\overline{q}_{h}}\right\|_V\leq \frac{1}{\lambda}\left\|u_{h \overline{g}_{h}\overline{q}_{h}}-z_d\right\|_H\leq c,\quad \forall h>0.
$$
Now, from the above estimations we obtain, when $h\rightarrow 0^+$:
\begin{equation*}
\exists f\in H\,:\,\, \overline{g}_{h} \rightarrow f \,\,\text{weakly in}\,\, H
\end{equation*}
\begin{equation*}
\exists \rho\in Q\,:\,\, \overline{q}_{h} \rightarrow \rho \,\,\text{weakly in}\,\, Q
\end{equation*}
\begin{equation*}
\exists \eta\in V\,:\,\, u_{h \overline{g}_{h}\overline{q}_{h}} \rightarrow \eta \,\,\text{weakly in}\,\, V (\text{in}\, H \,\text{strong})
\end{equation*}
\begin{equation*}
\exists \xi\in V\,:\,\, p_{h \overline{g}_{h}\overline{q}_{h}} \rightarrow \xi \,\,\text{weakly in}\,\, V (\text{in}\, H \,\text{strong}).
\end{equation*}
By using the above weak convergences, we can pas to the limit as $h\rightarrow 0^{+}$, and by uniqueness of the variational equalities (\ref{Pvariacional}) and (\ref{Padjvariacional}), we obtain that
$$
\eta=u_{f\rho},\quad \xi=p_{f\rho}.
$$
Next, by the weak lower semicontinuity of the functional $J_{h}$ and the uniqueness of the solution of the optimal control problem (\ref{OP}), we have that
$$
f=\overline{g}\quad \text{and}\quad \rho=\overline{q}.
$$
By the following inequalities 
\begin{equation*}
\begin{split}
\lambda\left\|u_{h \overline{g}_{h}\overline{q}_{h}}-u_{\overline{g}\,\overline{q}}\right\|_V^{2}&\leq (\overline{g}_{h}-\overline{g},u_{h \overline{g}_{h}\overline{q}_{h}}-b)_H-(\overline{q}_{h}-\overline{q},u_{h \overline{g}_{h}\overline{q}_{h}}-b)_Q\\&
+ (\overline{g},u_{\overline{g}\,\overline{q}}-u_{h \overline{g}_{h}\overline{q}_{h}})_H- (\overline{q},u_{\overline{g}\,\overline{q}}-u_{h \overline{g}_{h}\overline{q}_{h}})_Q
\end{split}
\end{equation*}
and
\begin{equation*}
\begin{split}
&\lambda\left\|p_{h \overline{g}_{h}\overline{q}_{h}}-p_{\overline{g}\,\overline{q}}\right\|_V^{2}\leq  a(p_{\overline{g}\,\overline{q}},p_{\overline{g}\,\overline{q}}-p_{h \overline{g}_{h}\overline{q}_{h}})_H- (p_{h \overline{g}_{h}\overline{q}_{h}},u_{\overline{g}\,\overline{q}}-u_{h \overline{g}_{h}\overline{q}_{h}})_Q
\end{split}
\end{equation*}
we obtain the strong convergences (\ref{Lim2}). Next, from the definition (\ref{Jh}), we have 
$$
\lim _{h \rightarrow 0^{+}}\left\|\overline{g}_{h}\right\|_H=\left\|\overline{g}\right\|_H\quad \text{and}\quad
\lim _{h \rightarrow 0^{+}}\left\|\overline{q}_{h}\right\|_Q=\left\|\overline{q}\right\|_Q
$$
and (\ref{Lim1}) holds. In a similar way, by using the elliptic variational equalities (\ref{Variacuhalfa}) and (\ref{Variacphalfa}), we prove (\ref{Lim3}) and (\ref{Lim4}). 

(ii) Following \cite{GaTa2}, we obtain that
$$
\left\|(\overline{g}_{h},\overline{q}_{h})-(\overline{g},\overline{q})\right\|_{H\times Q} \leq \frac{c_1\sqrt{\frac{1}{M_1^2}+\frac{||\gamma||^2}{M_2^2}}}{1-\frac{\sqrt{2}}{\lambda^2}\sqrt{\frac{1}{M_1^2}+\frac{||\gamma||^2}{M_2^2}}(1+||\gamma||)} h^{r-1}
$$
$$\left\|(\overline{g}_{h\alpha},\overline{q}_{h\alpha})-(\overline{g}_{\alpha},\overline{q}_{\alpha})\right\|_{H\times Q} \leq \frac{c_{1\alpha}\sqrt{\frac{1}{M_1^2}+\frac{||\gamma||^2}{M_2^2}}}{1-\frac{\sqrt{2}}{\lambda_{\alpha}^2}\sqrt{\frac{1}{M_1^2}+\frac{||\gamma||^2}{M_2^2}}(1+||\gamma||)} h^{r-1}
$$
where $c_1$ and $c_{1\alpha}$ are constants given in (\ref{Est1}) and (\ref{Est2}), respectively.
\end{proof}

\section{Convergence of the discrete optimal control pro\-blems ($P_{h\alpha}$) to ($P_{h}$)
 when $\alpha\rightarrow +\infty$}\label{Convergencealfa}

In this section, for each $h>0$, we obtain convergence results of the discrete simultaneous distributed-boundary optimal control problems $(P_{h\alpha})$ to $(P_h)$ when the parameter $\alpha \rightarrow +\infty$. For fixed $h>0$, we have the following convergences.
\begin{Lemma}
For fixed $(g,q) \in H\times Q$, $h>0$, we have the following limits:
\begin{equation}\label{limalfa1}
 \lim\limits_{\alpha \rightarrow+\infty}||u_{h\alpha gq}-u_{hgq}||_V=0.
\end{equation}
\begin{equation}\label{limalfa2}
\lim\limits_{\alpha \rightarrow+\infty}||p_{h \alpha gq}-p_{hgq}||_V=0.
\end{equation}
\end{Lemma}
\begin{proof}
For fixed $(g,q) \in H\times Q$, $h>0$, and by using the variational equalities (\ref{Variacuh}) and (\ref{Variacuhalfa}), and taking into account that for $\alpha >1$ we can split
$$
a_{\alpha}(u,v)=a_1(u,v)+(\alpha -1)\int_{\Gamma_1} u v d\gamma
$$
we obtain the following estimations
$$
||u_{h\alpha gq}-u_{hgq}||_V \leq c, \quad(\alpha-1) \int_{\Gamma_1}\left(u_{h\alpha gq}-b\right)^2 d\gamma \leq c, \quad \forall \alpha>1.
$$
From the above inequalities, we deduce that
$$
\exists \eta_{hgq} \in V / u_{h\alpha gq} \longrightarrow \eta_{hgq} \text { in } V \text { weakly (in } H \text { strong) as } \alpha \rightarrow+\infty \text { with } \eta_{hgq}\big|_{\Gamma_1}=b. 
$$
By using the variational equality (\ref{Variacuhalfa}), we can pass to the limit when $\alpha \rightarrow+\infty$, and by uniqueness of the variational equality (\ref{Variacuh}) we obtain that $\eta_{hgq}=u_{hgq}$. By using the above properties, and the variational equalities (\ref{Variacuh}) and (\ref{Variacuhalfa}), we deduce (\ref{limalfa1}) and by using a similar method we can obtain the limit (\ref{limalfa2}) for the discrete adjoint system state.
\end{proof}
\begin{Theorem}
We have the following limits:
\begin{equation}\label{limalfaop1}
\lim\limits_{\alpha \rightarrow+\infty}||(\overline{g}_{h\alpha},\overline{q}_{h\alpha})-(\overline{g}_{h},\overline{q}_{h})||_{H\times Q}=0.
\end{equation}
\begin{equation}\label{limalfaop2}
\lim\limits_{\alpha \rightarrow+\infty}||u_{h\alpha \overline{g}_{h\alpha}\overline{q}_{h\alpha}}-u_{h\overline{g}_{h}\overline{q}_{h}}||_V=0, \quad \lim\limits_{\alpha \rightarrow+\infty}||p_{h\alpha \overline{g}_{h\alpha}\overline{q}_{h\alpha}}-p_{h\overline{g}_{h}\overline{q}_{h}}||_V=0.
\end{equation}
\end{Theorem}
\begin{proof}
For each fixed $h>0$, the thesis holds in similar way that Theorem 7 in \cite{GaTa2}.
\end{proof}

\section{Double convergence of the discrete distributed-boundary optimal control problems ($P_{h\alpha}$) to ($P$)
 when $(h,\alpha)\rightarrow (0^+,+\infty)$}\label{Convergencehalfa}

In this section, we prove the main result of the paper.

\begin{Theorem}
We have the following limits:
\begin{equation}\label{limhalfaop1}
\lim\limits_{(h,\alpha) \rightarrow (0^+,+\infty)}||(\overline{g}_{h\alpha},\overline{q}_{h\alpha})-(\overline{g},\overline{q})||_{H\times Q}=0.
\end{equation}
\begin{equation}\label{limhalfaop2}
\lim\limits_{(h,\alpha) \rightarrow (0^+,+\infty)}||u_{h\alpha \overline{g}_{h\alpha}\overline{q}_{h\alpha}}-u_{\overline{g}\,\overline{q}}||_V=0, \quad \lim\limits_{(h,\alpha) \rightarrow (0^+,+\infty)}||p_{h\alpha \overline{g}_{h\alpha}\overline{q}_{h\alpha}}-p_{\overline{g}\,\overline{q}}||_V=0.
\end{equation}
\end{Theorem}
\begin{proof}
We obtain the proof in two steps.

\textbf{Step 1.} We show a sketch of the proof by obtaining the following estimations, for $h>0$ and $\alpha >1$:
\begin{equation*}
||u_{h00}||_V\leq c_1=b\sqrt{|\Omega|}
\end{equation*}
\begin{equation*}
||u_{h\alpha 00}||_V\leq c_2=\left(1+\frac{1}{\lambda_1}\right)c_1
\end{equation*}
\begin{equation*}
(\alpha - 1)\int_{\Gamma_1}(u_{h\alpha 00}-b)^2 d\gamma\leq c_3=\frac{c_1^2}{\lambda_1}
\end{equation*}
\begin{equation*}
||(\overline{g}_{h\alpha}, \overline{q}_{h\alpha})||_{H\times Q}\leq c_4=\frac{1}{\sqrt{\min\{M_1,M_2\}}}(c_2+||z_d||_H)
\end{equation*}
\begin{equation*}
||u_{h\alpha \overline{g}_{h\alpha}\overline{q}_{h\alpha}}||_H\leq c_5=c_2+2||z_d||_H
\end{equation*}
\begin{equation*}
||(\overline{g}_{h}, \overline{q}_{h})||_{H\times Q}\leq c_6=\frac{1}{\sqrt{\min\{M_1,M_2\}}}(c_1+||z_d||_H)
\end{equation*}
\begin{equation*}
||u_{h \overline{g}_{h}\overline{q}_{h}}||_V\leq c_7=\sqrt{2}(1+||\gamma||)c_6+c_1
\end{equation*}
\begin{equation*}
||u_{h\alpha \overline{g}_{h\alpha}\overline{q}_{h\alpha}}||_V\leq c_8=\sqrt{2}(1+||\gamma||)c_4+\left(1+\frac{1}{\lambda_1}\right)c_7
\end{equation*}
\begin{equation*}
(\alpha - 1)\int_{\Gamma_1}(u_{h\alpha \overline{g}_{h\alpha}\overline{q}_{h\alpha}}-b)^2 d\gamma\leq c_9=\frac{1}{\lambda_1}(\sqrt{2}(1+||\gamma||)c_4+c_7)^2
\end{equation*}
\begin{equation*}
||p_{h \overline{g}_{h}\overline{q}_{h}}||_V\leq c_{10}=\frac{1}{\lambda}(c_7+||z_d||_H)
\end{equation*}
\begin{equation*}
||p_{h\alpha \overline{g}_{h\alpha}\overline{q}_{h\alpha}}||_V\leq c_{11}=\frac{1}{\lambda_1}(c_8+||z_d||_H)+ \left(1+\frac{1}{\lambda_1}\right)c_{10}
\end{equation*}
\begin{equation*}
(\alpha - 1)\int_{\Gamma_1}(p_{h\alpha \overline{g}_{h\alpha}\overline{q}_{h\alpha}}-b)^2 d\gamma\leq c_{12}=\frac{1}{\lambda_1}(c_5+||z_d||_H+c_{10})^2.
\end{equation*}
Therefore, from the above estimations we have that
\begin{equation*}
\exists f\in H\,:\,\, \overline{g}_{h\alpha} \rightarrow f \,\,\text{weakly in}\,\, H,\,\, \text{when} \,\, (h,\alpha)\rightarrow(0^+,+\infty)
\end{equation*}
\begin{equation*}
\exists \rho\in Q\,:\,\, \overline{q}_{h\alpha} \rightarrow \rho \,\,\text{weakly in}\,\, Q,\,\, \text{when} \,\, (h,\alpha)\rightarrow(0^+,+\infty)
\end{equation*}
\begin{equation*}
\exists \eta\in V\,:\,\, u_{h\alpha \overline{g}_{h\alpha}\overline{q}_{h\alpha}} \rightarrow \eta \,\,\text{weakly in}\,\, V (\text{in}\, H \,\text{strong}),\,\, \text{when} \,\, (h,\alpha)\rightarrow(0^+,+\infty)
\end{equation*}
\begin{equation*}
\exists \xi\in V\,:\,\, p_{h\alpha \overline{g}_{h\alpha}\overline{q}_{h\alpha}} \rightarrow \xi \,\,\text{weakly in}\,\, V (\text{in}\, H \,\text{strong}),\,\, \text{when} \,\, (h,\alpha)\rightarrow(0^+,+\infty)
\end{equation*}
and 
\begin{equation*}
\exists f_h\in H\,:\,\, \overline{g}_{h\alpha} \rightarrow f_h \,\,\text{weakly in}\,\, H,\,\, \text{when} \,\, \alpha\rightarrow +\infty
\end{equation*}
\begin{equation*}
\exists \rho_h\in Q\,:\,\, \overline{q}_{h\alpha} \rightarrow \rho_h \,\,\text{weakly in}\,\, Q,\,\, \text{when} \,\, \alpha\rightarrow +\infty
\end{equation*}
\begin{equation*}
\exists \eta_h\in V\,:\,\, u_{h\alpha \overline{g}_{h\alpha}\overline{q}_{h\alpha}} \rightarrow \eta_h \,\,\text{weakly in}\,\, V (\text{in}\, H \,\text{strong}),\,\, \text{when} \,\, \alpha\rightarrow +\infty
\end{equation*}
\begin{equation*}
\exists \xi_h\in V\,:\,\, p_{h\alpha \overline{g}_{h\alpha}\overline{q}_{h\alpha}} \rightarrow \xi_h \,\,\text{weakly in}\,\, V (\text{in} \,H \,\text{strong}),\,\, \text{when} \,\, \alpha\rightarrow +\infty
\end{equation*}
and
\begin{equation*}
\exists f_{\alpha}\in H\,:\,\, \overline{g}_{h\alpha} \rightarrow f_{\alpha} \,\,\text{weakly in}\,\, H,\,\, \text{when} \,\, h\rightarrow 0^+
\end{equation*}
\begin{equation*}
\exists \rho_{\alpha}\in Q\,:\,\, \overline{q}_{h\alpha} \rightarrow \rho_{\alpha} \,\,\text{weakly in}\,\, Q,\,\, \text{when} \,\, h\rightarrow 0^+
\end{equation*}
\begin{equation*}
\exists \eta_{\alpha}\in V\,:\,\, u_{h\alpha \overline{g}_{h\alpha}\overline{q}_{h\alpha}} \rightarrow \eta_{\alpha} \,\,\text{weakly in}\,\, V (\text{in} \,H \,\text{strong}),\,\, \text{when} \,\, h\rightarrow 0^+
\end{equation*}
\begin{equation*}
\exists \xi_{\alpha}\in V\,:\,\, p_{h\alpha \overline{g}_{h\alpha}\overline{q}_{h\alpha}} \rightarrow \xi_{\alpha} \,\,\text{weakly in}\,\, V (\text{in} \,H \,\text{strong}),\,\, \text{when} \,\, h\rightarrow 0^+.
\end{equation*}
\textbf{Step 2.}
Now, taking into account that 
$$\eta=\eta_h=b\,\,\text{on}\,\,\Gamma_1,$$
$$\xi=\xi_h=0\,\,\text{on}\,\,\Gamma_1,$$
by the uniqueness of the solutions of the simultaneous distributed-boundary optimal control problems $(P_{h\alpha})$, $(P_{h})$, $(P_{\alpha})$ and $(P)$, and the uniqueness of the solutions of the elliptic variational equalities corresponding to their state systems, we obtain that
$$
\eta_h=u_{h f_{h} \rho_h}=u_{h \overline{g}_h \overline{q}_h},\quad \xi_h=p_{h f_{h} \rho_h}=p_{h \overline{g}_h \overline{q}_h},\quad f_h=\overline{g}_h,\quad \rho_h=\overline{q}_h
$$
$$
\eta_{\alpha}=u_{\alpha f_{\alpha} \rho_{\alpha}}=u_{\alpha \overline{g}_{\alpha} \overline{q}_{\alpha}},\quad \xi_{\alpha}=p_{\alpha f_{\alpha} \rho_{\alpha}}=p_{\alpha \overline{g}_{\alpha} \overline{q}_{\alpha}},\quad f_{\alpha}=\overline{g}_{\alpha},\quad \rho_{\alpha}=\overline{q}_{\alpha}
$$
and the limits (\ref{Lim1}) and (\ref{Lim2}).
Next, by using \cite{GaTa2}, we get
$$
\lim_{\alpha \rightarrow +\infty}||f_{\alpha}-\overline{g}||_H=0,\quad \lim_{\alpha \rightarrow +\infty}||\rho_{\alpha}-\overline{q}||_Q=0
$$
$$
\lim_{\alpha \rightarrow +\infty}||\eta_{\alpha}-u_{\overline{g}\,\overline{q}}||_V=0,\quad \lim_{\alpha \rightarrow +\infty}||\xi_{\alpha}-p_{\overline{g}\,\overline{q}}||_V=0
$$
and therefore the double limits (\ref{limhalfaop1}) and (\ref{limhalfaop2}) holds, when $(h,\alpha)\rightarrow (0^+,+\infty)$.
\end{proof}

\section{Relationship among the optimal values corresponding to the optimal control problems  $(P)$, $(P_{h})$, $(P_{\alpha})$ and $(P_{h\alpha})$}

In this section, we obtain the estimates given below.
\begin{Lemma}
If $M_1$ and $M_2$ satisfy the inequalities (\ref{puntofijo}) and the continuous system states and adjoint states have the regularity $u_{\overline{g}\,\overline{q}}, u_{\alpha \overline{g}_{\alpha}\overline{q}_{\alpha}}, p_{\overline{g}\,\overline{q}}, p_{\alpha \overline{g}_{\alpha}\overline{q}_{\alpha}} \in H^r(\Omega)$ $(1<r \leq 2)$, we have the following error bonds:
\begin{equation}\label{cota1}
0\leq J(\overline{g}_h,\overline{q}_h)-J(\overline{g},\overline{q}) \leq c h^{2(r-1)}
\end{equation}
\begin{equation}\label{cota2}
0\leq J_h(\overline{g},\overline{q})-J_h(\overline{g}_h,\overline{q}_h) \leq c h^{2(r-1)}
\end{equation}
\begin{equation}\label{cota3}
J(\overline{g},\overline{q})-J_h(\overline{g}_h,\overline{q}_h) \leq c h^{r-1}
\end{equation}
and, for each $\alpha>0$:
\begin{equation}\label{cota4}
0\leq J_{\alpha}(\overline{g}_{h\alpha},\overline{q}_{h\alpha})-J_{\alpha}(\overline{g}_{\alpha},\overline{q}_{\alpha}) \leq c_{\alpha} h^{2(r-1)}
\end{equation}
\begin{equation}\label{cota5}
0\leq J_{h\alpha}(\overline{g}_{\alpha},\overline{q}_{\alpha})-J_{h\alpha}(\overline{g}_{h\alpha},\overline{q}_{h\alpha}) \leq c_{\alpha} h^{2(r-1)}
\end{equation}
\begin{equation}\label{cota6}
J_{\alpha}(\overline{g}_{\alpha},\overline{q}_{\alpha}) - J_{h\alpha}(\overline{g}_{h\alpha},\overline{q}_{h\alpha}) \leq c_{\alpha} h^{r-1}
\end{equation}
where $c$ and $c_{\alpha}$ are different constants independents of $h$.
\end{Lemma}
\begin{proof}
Estimations (\ref{cota1}),(\ref{cota2}), (\ref{cota4}) and (\ref{cota5}) follow from the estimations (\ref{Est1}), (\ref{Est2}), (\ref{cotaop2}), (\ref{cotaop4}) and the equalities:
$$J(\overline{g}_h,\overline{q}_h)-J(\overline{g},\overline{q})=\frac{1}{2}\left\|u_{\overline{g}_h\overline{q}_h}-u_{\overline{g}\,\overline{q}}\right\|_H^2+\frac{M_1}{2}\left\|\overline{g}_h-\overline{g}\right\|_H^2+\frac{M_2}{2}\left\|\overline{q}_h-\overline{q}\right\|_Q^2$$
$$J_h(\overline{g},\overline{q})-J_h(\overline{g}_h,\overline{q}_h)=\frac{1}{2}\left\|u_{h\overline{g}\,\overline{q}}-u_{h\overline{g}_h\overline{q}_h}\right\|_H^2+\frac{M_1}{2}\left\|\overline{g}-\overline{g}_h\right\|_H^2+\frac{M_2}{2}\left\|\overline{q}-\overline{q}_h\right\|_Q^2$$
$$J_{\alpha}(\overline{g}_{h\alpha},\overline{q}_{h\alpha})-J_{\alpha}(\overline{g}_{\alpha},\overline{q}_{\alpha})=\frac{1}{2}\left\|u_{h\alpha\overline{g}_{h\alpha}\overline{q}_{h\alpha}}-u_{\alpha\overline{g}_{\alpha}\,\overline{q}_{\alpha}}\right\|_H^2+\frac{M_1}{2}\left\|\overline{g}_{h\alpha}-\overline{g}_{\alpha}\right\|_H^2+\frac{M_2}{2}\left\|\overline{q}_{h\alpha}-\overline{q}_{\alpha}\right\|_Q^2$$
$$J_{h\alpha}(\overline{g}_{\alpha},\overline{q}_{\alpha})-J_{h\alpha}(\overline{g}_{h\alpha},\overline{q}_{h\alpha})=\frac{1}{2}\left\|u_{h\alpha\overline{g}_{\alpha}\,\overline{q}_{\alpha}}-u_{h\alpha\overline{g}_{h\alpha}\overline{q}_{h\alpha}}\right\|_H^2+\frac{M_1}{2}\left\|\overline{g}_{\alpha}-\overline{g}_{h\alpha}\right\|_H^2+\frac{M_2}{2}\left\|\overline{q}_{\alpha}-\overline{q}_{h\alpha}\right\|_Q^2.$$
Estimations (\ref{cota3}) and (\ref{cota6}) follow from estimations (\ref{Est1}) and (\ref{Est2}), taking into account that 
\begin{equation*}
\begin{split}
J(\overline{g},\overline{q})-J_h(\overline{g}_h,\overline{q}_h)&\leq J(\overline{g}_h,\overline{q}_h)-J_h(\overline{g}_h,\overline{q}_h)\\&=\frac{1}{2}\left(\left\|u_{\overline{g}_h\overline{q}_h}-z_d\right\|_H^2-\left\|u_{h\overline{g}_h\overline{q}_h}-z_d\right\|_H^2\right)\\ &= \frac{1}{2}\left(u_{\overline{g}_h\overline{q}_h}-u_{h\overline{g}_h\overline{q}_h},u_{h\overline{g}_h\overline{q}_h}+u_{\overline{g}_h\overline{q}_h}-2z_d\right)_H\\ &\leq \frac{1}{2}\left\|u_{\overline{g}_h\overline{q}_h}- u_{h\overline{g}_h\overline{q}_h}\right\|_H \left(\left\|u_{h\overline{g}_h\overline{q}_h}-z_d\right\|_H+\left\|u_{\overline{g}_h\overline{q}_h}-z_d\right\|_H\right)\\&\leq ch^{r-1}
\end{split}
\end{equation*}
and
\begin{equation*}
\begin{split}
&J_{\alpha}(\overline{g}_{\alpha},\overline{q}_{\alpha})-J_{h\alpha}(\overline{g}_{h\alpha},\overline{q}_{h\alpha}) \leq\frac{1}{2}\left(\left\|u_{\alpha\overline{g}_{h\alpha}\overline{q}_{h\alpha}}-z_d\right\|_H^2-\left\|u_{h\alpha\overline{g}_{h\alpha}\overline{q}_{h\alpha}}-z_d\right\|_H^2\right)\\ &\leq \frac{1}{2}\left\|u_{\alpha\overline{g}_{h\alpha}\overline{q}_{h\alpha}}- u_{h\alpha\overline{g}_{h\alpha}\overline{q}_{h\alpha}}\right\|_H \left(\left\|u_{h\alpha\overline{g}_{h\alpha}\overline{q}_{h\alpha}}-z_d\right\|_H+\left\|u_{\alpha\overline{g}_{h\alpha}\overline{q}_{h\alpha}}-z_d\right\|_H\right) \leq c_{\alpha}h^{r-1}.
\end{split}
\end{equation*}
\end{proof}

\begin{Remark}
In a forthcoming paper, we will do the numerical analysis and its corresponding error estimates when we replace la condition (\ref{Palfa}ii) by the following 
$$
-\frac{\partial u}{\partial n}|_{\Gamma_{1}} \in \alpha\partial j(u).
$$	
Here $j(x,.)$ is locally Lipschitz for a.e. $x\in \Gamma_{1}$ and not necessary differentiable following \cite{BGT, GMOT, GaTa3, HS}. Therefore, the variational formulation, for the system state, will be given by an elliptic hemivariational inequality, and the corresponding control variable can be the energy $g$, or the heat flux $q$ or the vectorial control $(g,q)$.
\end{Remark}

\section{Conclusions}

For two vectorial continuous optimal control problems $(P_{\alpha})$ and $(P)$, and for the corresponding two vectorial discrete optimal control problems $(P_{h\alpha})$ and $(P_{h})$ we have obtained a commutative diagram when $h\rightarrow 0^{+}$ and $\alpha\rightarrow +\infty$, with $\alpha\rightarrow +\infty$ and $h\rightarrow 0^{+}$, and the corresponding double convergence when $(h,\alpha)\rightarrow (0^{+},+\infty)$ simultaneously for the optimal controls, for the optimal system states and for the optimal adjoint states. The parameter $\alpha$ can be considered as the heat transfer coefficient on a portion of the boundary of a material, and $h$ is the parameter of the finite element approximation.

\section*{Acknowledgements}

The present work has been partially sponsored by the European Union’s Horizon 2020 Research and Innovation Programme under the Marie Sklodowska-Curie grant agreement 823731 CONMECH and by the Project PIP No. 0275 from CONICET and Universidad Austral, Rosario, Argentina for the third author, and by the Project PPI No. 18/C555 from SECyT-UNRC, Río Cuarto, Argentina for the first and second authors.


\end{document}